\documentclass[12pt]{article}

\usepackage{amsfonts}
\usepackage{amssymb}
\usepackage{amsmath}
\usepackage{xcolor}
\usepackage[all]{xy}
\usepackage[utf8]{inputenc}

\def\Mod{\hbox{-}\mathsf{Mod}}
\def\uno{\mathbf{1}}
\def\mpn{\medskip\par\noindent}
\def\dom{\backslash}

\usepackage[left=1.20in,right=1.20in]{geometry}
\definecolor{ao(english)}{rgb}{0.0, 0.5, 0.0}
\definecolor{brickred}{rgb}{0.8, 0.25, 0.33}
\definecolor{burntorange}{rgb}{0.8, 0.33, 0.0}
\definecolor{beaver}{rgb}{0.62, 0.51, 0.44}
\definecolor{brown(traditional)}{rgb}{0.59, 0.29, 0.0}
\definecolor{ao(english)}{rgb}{0.0, 0.5, 0.0}
\definecolor{verde}{rgb}{0.12, 0.8, 0.17}
\def\rouge{\color{red}}

\def\sur{\overline}

\def\CH{\mathcal{H}}
\newcommand{\lae}{\varepsilon}

\newcommand{\bind}{\textrm{Ind}}
\newcommand{\bres}{\textrm{Res}}

\newcommand{\biso}{\textrm{Iso}}

\newcommand{\Ind}{\mathrm{Ind}}
\newcommand{\Res}{\mathrm{Res}}

\newcommand{\Ker}{\mathrm{Ker}}

\newtheorem{teo}{Theorem}[section]
\newtheorem{prop}[teo]{Proposition}

\newtheorem{lema}[teo]{Lemma}

\newtheorem{nota}[teo]{Notation}

\newtheorem{rem}[teo]{Remark}

\def\C{\mathbb{C}}

\def\Z{\mathbb{Z}}
\def\HH{\mathcal{H}H}
\def\un{\mathbf{1}}
\def\mpn{\medskip\par\noindent}

\def\hZt{\widehat{\Z}^\times}
\def\Loc{\mathcal{L}oc}
\def\sur{\overline}

\def\N{\mathbb{N}}
\def\pf{\noindent{\bf Proof:} }
\def\mpn{\medskip\par\noindent}
\def\endpf{~\leaders\hbox to 1em{\hss\  \hss}\hfill~\raisebox{.5ex}{\framebox[1ex]{}}\smallskip\par}

\reversemarginpar

\title{On the separability of some Green biset functors}
\author{Serge Bouc and Nadia Romero}

\begin{document}

\maketitle
\begin{abstract}
We show that the Green biset functor $R_{\mathbb{C}}$ of complex characters over $\mathbb{Z}$, is not separable, i.e. it is not projective as a bimodule over itself. Also, we show that $RB_G$, the Burnside biset functor shifted by a finite group $G$, over a commutative ring $R$, is separable if and only if $|G|$ is invertible in $R$. Finally, to address the question of the relation between functors and their evaluations, we show that the Burnside $R$-algebra $RB(G)$ is separable if and only if $|G|$ is invertible in $R$.\medskip

\end{abstract}

\section{Introduction}
Let $R$ be a unital commutative ring and $A$ be an (associative) $R$-algebra. Among the various definitions of $A$ being {\em separable} (over $R$), one of them (\cite[Proposition 1.1]{demeyer-ingraham}) is that $A$ be projective as an $(A,A)$-bimodule, or equivalently, that the product map 
$$\mu_A:A\otimes_RA\to A$$ be split surjective. This definition allows for generalizations to other contexts: For example, in~\cite[Definition 5.5]{hoch}, the second author introduces the notion of separability for a monoid $A$ in the category $\mathcal{F}$ of $R$-linear functors from an $R$-linear monoidal category $\mathcal{X}$ to the category of $R$-modules. \par
Here, we consider the special case where $\mathcal{X}$ is the biset category of finite groups (\cite[Definition 3.1.1]{biset}), for its monoidal structure given by the direct product of finite groups. In this case, a monoid in $\mathcal{F}$ is called a {\em Green biset functor} over $R$. We give examples of two classical Green biset functors, the functor $R_\C$ of complex characters over $\Z$ and the shifted Burnside functors $RB_G$, where $G$ is a fixed finite group. We prove (Proposition~\ref{R_C non-separable}) that $R_\C$ is not separable, and (Theorem~\ref{sep}) that $RB_G$ is separable if and only if the order of $G$ is invertible in $R$. Since the question of separability requires the use of the tensor product of functors, in Section \ref{prelim} we  recall and state some properties of it, which we will need to treat the case of $RB_G$. Some of these properties are actually well known so, since their nature is rather technical, we state them without a proof.

We conclude (Section~\ref{evaluations}) by comparing in some examples the separability of a Green functor $A$ and the separability of its evaluations. In particular, we show  (Theorem~\ref{burnside} for $R=\Z$) that whereas the Burnside functor $B$ is always separable (since the multiplication morphism $B\otimes B\to B$ is an isomorphism), the Burnside ring $B(G)$ of a non-trivial finite group $G$ is not separable.  However, we also show  (Proposition~\ref{burnside2} for $R=\Z$) that the first Hochschild cohomology group $HH^1\big(B(G),B(G)\big)$ is always trivial.\medskip\par

\section{Preliminaries}
\label{prelim}


 In what follows, $R$ is a commutative ring with unity, denoted by 1. The trivial group is denoted by $\uno$. For a finite group $G$, we denote by $B(G)$ its {\em Burnside group}, i.e. the Grothendieck group of finite $G$-sets. The cartesian product of $G$-sets endows $B(G)$ with a structure of commutative ring, called the {\em Burnside ring} of $G$. The {\em Burnside algebra} $RB(G)$ of $G$ over $R$ is the tensor product $R\otimes _\Z B(G)$. As an $R$-module, it has a basis consisting of the elements $[G/H]=1\otimes_\Z[G/H]$, for $H$ in a set $[s_G]$ of representatives of conjugacy classes of subgroups of $G$.\par
For a subgroup $H$ of $G$, we denote by $\Phi_H:\alpha\mapsto |\alpha^H|$ the unique $R$-linear map from $RB(G)$ to $R$ sending a finite $G$-set $X$ to the cardinality of the set $X^H$ of $H$-fixed points on $X$. The map $\Phi_H$ is a ring homomorphism. \par
When the order of $G$ is invertible in $R$, we have idempotents (\cite{yoshidaidemp}, \cite{gluck})
$$e_H^G=\frac{1}{|N_G(H)|}\sum_{K\leqslant H}|K|\mu(K,H)[G/K]$$
in the algebra $RB(G)$, indexed by subgroups $H$ of $G$, up to conjugation. The idempotents $e_H^G$, for $H\in [s_G]$, are orthogonal, and their sum is equal to the idendity element $[G/G]$ of $RB(G)$. The idempotent $e_H^G$ is characterized by the fact that $e_H^G\cdot \alpha=\Phi_H(\alpha)e_H^G$ for any $\alpha\in RB(G)$.\medskip\par

Recall that the biset category with coefficients in $R$ is the category $R\mathcal{C}$ having as objects all finite groups and as set of arrows from a group $G$ to a group $H$, the Burnside group with coefficients in $R$ of $H\times G$, which we denote by $RB(H,\, G)$. The reason for this notation is that we think of finite $(H\times G)$-sets as $(H,\, G)$\textit{-bisets}. We invite the reader to take a look at Chapter 2 in \cite{biset} for the definition of bisets and their composition $\circ$ in $R\mathcal{C}$. The category of $R$-linear functors from $\mathcal{C}$ to $R$-Mod, denoted by $\mathcal{F}_{\mathcal{C},R}$, is called the category of biset functors.

We recall the notation of  some of the basic bisets we will use throughout the paper. Let $G$ be a finite group, $H$ be a subgroup of $G$ and $K$ be a group isomorphic to $G$.
\begin{itemize}
\item[$\bullet$] The induction, $\bind_H^G$, is the natural $(G,\,H)$-biset $G$.
\item[$\bullet$] The restriction, $\bres_H^G$, is the natural $(H,\,G)$-biset $G$. 
\item[$\bullet$] The isomorphism, $\biso^{G}_{K}$, is the natural $(K,\,G)$-biset $G$.
\end{itemize}

When applying the Burnside functor to any of these arrows, say the induction, instead of writing $RB(\bind_H^G)$, we will simply write $\bind_H^G$.

\begin{nota}
\begin{enumerate}
\item Let $G$ be a finite group. Since the identity arrow for the object $G$ in the biset category is the (class of the) $(G,\, G)$-biset $G$ in $RB(G,\,G)$, we denote this arrow simply as $G$. 
\item Recall  that given finite groups $H$, $H'$, $K$, $K'$, a  finite $(H,\, K)$-biset $X$ and a finite $(H',\, K')$-biset, the product $X\times Y$ has a natural structure of $(H\times H',\, K\times K')$-biset. This defines a bilinear map
\begin{displaymath}
RB(H,\, K)\times RB(H',\, K')\rightarrow RB(H\times H',\, K\times K'),
\end{displaymath}
which we continue to denote by $(\alpha,\, \beta)\mapsto \alpha\times \beta$. For more properties on this product see Chapter 8 in \cite{biset}.
\item Given a family of finite groups $G_1,\ldots ,G_m$ we may abbreviate $G_1\times \cdots\times G_m$ as $G_1\cdots G_m$. If $G_1=\cdots =G_m=G$, we simply note this product as $G^m$.
\end{enumerate}
\end{nota}

The direct product of groups and the product of bisets, just defined, make the biset category a symmetric monoidal category (see Lemma 8.1.2 in \cite{biset}). Actually, it is an essentially small symmetric monoidal category, enriched in $R$-Mod. So, $\mathcal{F}_{\mathcal{C},R}$ is also endowed with a tensor product structure $\otimes$, given by the Day convolution (see \cite{libro}, for instance). With this, $\mathcal{F}_{\mathcal{C},R}$ becomes an abelian, symmetric monoidal, closed category  with identity given by the Burnside functor $RB(\,\_\,,\uno)$. A monoid in $\mathcal{F}_{\mathcal{C},R}$ is called a Green biset functor. This means that a Green biset functor is an object $A$ in $\mathcal{F}_{\mathcal{C},R}$, together with morphisms $\mu_A:A\otimes A\rightarrow A$ and $e_A:RB\rightarrow A$ satisfying obvious associativity and identity diagrams, as stated in Definition 8.5.1 in \cite{biset}. An equivalent way of defining a Green biset functor is given right after Definition 8.5.1 in \cite{biset}, in terms of \textit{bilinear products}. We will use both  definitions, depending on the situation. The same applies for modules over Green biset functors (see Definition 8.5.5 and the paragraph after it).

\begin{rem}
\label{adj-sim}
 As a corollary of Lemma 8.1.2 in \cite{biset}, for any finite group $G$, we have an $R$-linear functor
\begin{displaymath}
\,\_\,\times G:R\mathcal{C}\rightarrow R\mathcal{C},
\end{displaymath}
which is self-adjoint. Indeed, given $H$, $G$ and $K$ finite groups, the $R$-modules
\begin{displaymath}
RB(H\times G,\, K)\quad\textrm{and}\quad RB(H,\, K\times G)
\end{displaymath}
identify both with $RB(H\times G\times K)$. Hence, a finite $(H\times G\times K)$-set can be seen as an $(H\times G,\, K)$-biset, as an $(H,\, G\times K)$-biset or as an $(H,\, K\times G)$-biset, depending on the situation. 
\end{rem}

We will be dealing with two examples of Green biset functors, the functor of complex characters with integer coefficients $R_{\mathbb{C}}$, and the shifted Burnside biset functor $RB_G$, for a fixed finite group $G$.

Regarding $R_{\mathbb{C}}$, recall that it is the Green biset functor sending a finite group $G$ to the group $R_\C(G)$ of its complex characters (or equivalently, the Grothendieck group of the category of finitely generated $\C G$-modules). The biset operations are induced by tensoring with permutation bimodules: If $H$ is a finite group, and $U$ is a finite $(H,G)$-biset, then the functor $M\mapsto \C U\otimes_{\C G}M$ from finitely generated $\C G$-modules to finitely generated $\C H$-modules induces a linear map $R_\C(U):R_\C(G)\to R_\C(H)$. Here $\C U$ is the $\C$-vector space with basis $U$, viewed as a $(\C H,\C G)$-bimodule. \par
The Green functor structure on $R_\C$ is induced by the external tensor product: If $G$ and $H$ are finite groups, if $M$ is a finitely  generated $\C G$-module and $N$ is a finitely generated $\C H$-module, let $M\boxtimes N$ denote the tensor product $M\otimes_\C N$, endowed with its obvious structure of $\C(G\times H)$-module. This induces a product 
$$(\chi,\rho)\in R_\C(G)\times R_\C(H)\mapsto \chi\times\rho\in R_\C(G\times H),$$
and one checks easily that we get in this way a Green biset functor structure on $R_\C$. The identity element $\epsilon_{R_\C}\in R_\C(\un)$ is the class of the trivial module $\C$ for the trivial group $\un$, i.e. the trivial character of the trivial group. \par

Given a finite group $G$, the shifted functor $RB_G$ is defined by
\begin{displaymath}
RB_G(H)=RB(H\times G)\quad\textrm{and}\quad RB_G(\alpha)=RB(\alpha\times G)
\end{displaymath}
for $H$ a finite group and $\alpha$ and arrow in $R\mathcal{C}$. It is a Green biset functor with the product
\begin{displaymath}
RB_G(H)\times RB_G(K)\rightarrow RB_G(H\times K)
\end{displaymath}
sending an $HG$-set $X$ and a $KG$-set $Y$ to the $HKG$-set $X\times Y$ with \textit{diagonal} action of $G$. What we mean by diagonal action of $G$ is the following:
\begin{displaymath}
(h,\,k,\,g)(x,\, y)=((h,\,g)x,\, (k,\,g)y)
\end{displaymath}
for all $(h,\,k,\,g)\in HKG$ and $(x,\,y)\in X\times Y$.  The identity element $\varepsilon$ in $RB_G(\uno)$ is the trivial $G$-set $\{\bullet\}$.

\begin{nota}
For  $a\in RB(H\times G)$ and $b\in RB(K\times G)$, their product in $RB(HKG)$, obtained by extending linearly the previous diagonal construction, is denoted by $a\times^Gb$. Using the symmetry in $R\mathcal{C}$, this construction is well defined no matter where the $G$ is in the product. That is, we may apply it to $a\in RB(H\times G)$ and $b\in RB(G\times K)$, and the result can be seen in $RB(HGK)$ or in $RB(HKG)$ if it is convenient.  We will use the same notation in all possible cases, taking care there is no risk of confusion.

We denote by $\Delta (G)$ the subgroup $\{(g,\,g)\mid g\in G\}$ of $G\times G$.
\end{nota}

With $X$ and $Y$ as above, we see that if we consider first $X\times Y$ as an $HGKG$-set, then
\begin{displaymath}
X\times^GY=\biso^{HK\Delta(G)}_{HKG}\circ\bres^{HKGG}_{HK\Delta(G)}\circ \biso^{HGKG}_{HKGG}(X\times Y).
\end{displaymath}
Usually we will omit the isomorphism between $\Delta(G)$ and $G$.

To work with the functor $RB_G$ we also need to recall some general facts concerning the tensor product of biset functors. We begin with the following description, appearing in Section 8.4 of \cite{biset}.

\begin{rem}
\label{ten-obj}
Let $M$ and $N$ be biset functors, $G$ and $H$ be finite groups and $\varphi\in RB(H,\,G)$. Then
\begin{displaymath}
(M\otimes N)(G)=\left(\bigoplus_{D,D',f}(M(D)\otimes_R N(D'))\right)/\mathcal{R},
\end{displaymath}
where $D$ and $D'$ run through a given set $\mathcal{S}$ of representatives of isomorphism classes of groups and $f\in RB(G,\, D\times D')$. An element of the form $m\otimes n\in M(D)\otimes_R N(D')$, for a summand indexed by $(D,D',f)$, is denoted by $[m\otimes n]_{D,D',f}$. With this notation, $\mathcal{R}$ is the $R$-submodule of $\bigoplus_{\substack{D,D',f}}(M(D)\otimes_R N(D'))$ generated by elements of the form
\begin{displaymath}
[m\otimes n]_{D,D',rf+r'f'}-(r[m\otimes n]_{D,D',f}+r'[m\otimes n]_{D,D',f'})\textrm{ and }
\end{displaymath}
\begin{displaymath}
[m\otimes n]_{D,D',f_1\circ(\alpha\times\beta)}-[M(\alpha)(m)\otimes N(\beta)(n)]_{D_1,D_1',f_1},
\end{displaymath}
for $r$, $r'\in R$, $f'\in RB(G,\, D\times D')$, $D_1$ and $D'_1$ groups in $\mathcal{S}$, $\alpha\in RB(D_1,\,D)$, $\beta\in RB(D_1',\,D')$ and $f_1\in RB(G,\, D_1\times D'_1)$.

Also, $(M\otimes N)(\varphi):(M\otimes N)(G)\rightarrow (M\otimes N)(H)$ sends the class of $[m\otimes n]_{D,D',f}$ to the class of $[m\otimes n]_{D,D',\varphi\circ f}$.
\end{rem}

In what follows we will  also use the notation $[m\otimes n]_{D,D',f}$ for the class of $[m\otimes n]_{D,D',f}$ in the quotient  by $\mathcal{R}$. Also, $D$, $D'$, $G$, $H$ and $K$ denote finite groups.

\begin{rem}
\label{ten-fle}
By Remark 8.4.3 in \cite{biset}, we know that the arrows from a tensor product of biset functors, $M\otimes N$, to another biset functor $P$, are in one-to-one correspondence with the set of {\em bilinear pairings} from $M$, $N$ to $P$, i.e. the set of natural transformations from the bifunctor $(G,H)\mapsto M(G)\times N(H)$  from $R\mathcal{C}\times R\mathcal{C}$ to $R\Mod$ to the bifunctor $(G,H)\mapsto P(G\times H)$. Such a bilinear pairing consists of a family $(\phi_{G,H})$ of bilinear maps
\begin{displaymath}
\phi_{G,\, H}:M(G)\times N(H)\rightarrow P(G\times H),
\end{displaymath}
satisfying obvious functoriality conditions. 
\end{rem}


The relation between these bilinear pairings and natural transformations is the follo\-wing: Given a bilinear pairing $\phi$ as in the previous remark, the corresponding natural transformation $M\otimes N\rightarrow P$ is given by
\begin{displaymath}
(M\otimes N)(K)\rightarrow P(K)\quad [m\otimes n]_{D,D',f}\mapsto P(f)\big(\phi_{D,\,D'}(m,\,n)\big),
\end{displaymath}
for $f:D\times D'\rightarrow K$. On the other hand, given a natural transformation $\psi:M\otimes N\rightarrow P$, the corresponding bilinear pairing is given by
\begin{displaymath}
M(G)\times N(H)\rightarrow P(G\times H)\quad (m,\, n)\mapsto  \psi_{G\times H}\big([m\otimes n]_{G,H,G\times H}\big).
\end{displaymath}
In particular, whenever we have two natural transformations $t:M\rightarrow M_1$ and $u:N\rightarrow N_1$ between biset functors, we have:
\begin{displaymath}
(t\otimes u)_K:(M\otimes N)(K)\rightarrow (M_1\otimes N_1)(K), \quad [m\otimes n]_{D,D',f}\mapsto \big[t_D(m)\otimes u_{D'}(n)\big]_{D,D',f}
\end{displaymath} 
for $f:D\times D'\rightarrow K$. The bilinear pairings are given by
\begin{displaymath}
M(G)\times N(H)\rightarrow (M_1\otimes N_1)(G\times H),\quad (m,\, n)\mapsto \big[t_G(m)\otimes u_{H}(n)\big]_{G,H,G\times H}.
\end{displaymath}

As a corollary of these observations we have the following lemma.

\begin{lema}
\label{acciones}
Let $A$ be a Green biset functor and $M$ be an $A$-module. If the action of $A$ on $M$ is given by the bilinear map
\begin{displaymath}
A(G)\times M(H)\mapsto M(G\times H)\quad (a,\, m)\mapsto a\times m\textrm{ for }a\in A(G), m\in M(H),
\end{displaymath}
then the natural transformation $A\otimes M\rightarrow M$ is given by 
\begin{displaymath}
(A\otimes M)(K)\rightarrow M(K)\quad [b\otimes n]_{D,D',f}\mapsto M(f)(b\times n) 
\end{displaymath}
for $f:D\times D'\rightarrow K$, $b\in A(D)$ and $n\in M(D')$. In particular, the product of $A$ is given by
\begin{displaymath}
(\mu_A)_K:(A\otimes A)(K)\rightarrow A(K)\quad [a\otimes b]_{D,D',f}\mapsto A(f)(a\times b) 
\end{displaymath}
for $f:D\times D'\rightarrow K$, $a\in A(D)$ and $b\in M(D')$.
\end{lema}

Since the category of biset functors is monoidal, the following lemma holds. Nevertheless, in the next section we will need explicit functions for the isomorphism. The proof is a bit long and tedious, so we omit it for the sake of simplicity.

\begin{lema}
\label{asso}
Let $M$, $N$ and $P$ be biset functors, the associativity of the tensor product $(M\otimes N)\otimes P\cong M\otimes (N\otimes P)$ is given by the arrows
\begin{displaymath}
\big((M\otimes N)\otimes P\big)(K)\rightarrow \big(M\otimes (N\otimes P)\big)(K)
\end{displaymath}
\begin{displaymath}
 \big[[m\otimes n]_{X_1,Y_1,f_1}\otimes l\big]_{X,Y,f}\mapsto\big[m\otimes[n\otimes l]_{Y_1,Y,Y_1Y}\big]_{X_1, Y_1Y, f\circ(f_1Y)}
\end{displaymath}
and
\begin{displaymath}
\big(M\otimes (N\otimes P)\big)(K)\rightarrow \big((M\otimes N)\otimes P\big)(K)
\end{displaymath}
\begin{displaymath}
 \big[m\otimes [n\otimes l]_{X_2,Y_2,f_2}\big]_{X,Y,f}\mapsto\big[[m\otimes n]_{X,X_2,XX_2}\otimes l\big]_{XX_2, Y_2, f\circ(Xf_2)},
\end{displaymath}
for $X$, $Y$, $X_i$, $Y_i$, with $i=1,\,2$, finite groups, $f_1:X_1 Y_1\rightarrow X$ $f_2:X_2Y_2\rightarrow Y$ and $f:X Y\rightarrow K$.
\end{lema}


\section{Separability of certain Green biset functors}

\subsection{$R_\C$ is not separable}

We will show that the Green biset functor $R_\C$ is not separable by showing that there exists an $(R_\C,R_\C)$-bimodule $L$ such that the first cohomology functor $\HH^1(R_\C,L)$ is non-zero. For this, we need the following notation:
\begin{nota} \begin{enumerate}
\item We denote by $\hZt$ the inverse limit for $n\in \N$ of the unit groups $(\Z/n\Z)^\times$ of the rings $\Z/ n\Z$, for the projection maps $\pi_{n,m}:(\Z/n\Z)^\times\to (\Z/m\Z)^\times$, when $m|n$. We denote by $\pi_n$ the natural projection $\hZt\to (\Z/n\Z)^\times$.
\item For a finite group $G$, we denote by $\Loc(G)$ the set of locally constant maps $\hZt\to R_\C(G)$, i.e. the set of maps $f:\hZt\to R_\C(G)$ such that there exist $n\in \N$ and $\sur{f}:(\Z/n\Z)^\times\to R_\C(G)$ with $f=\sur{f}\circ \pi_n$.
\end{enumerate}
\end{nota}
Recall (\cite{biset} Section 7.2) that $\hZt$ acts by automorphisms on the Green biset functor $R_\C$: if $s\in \hZt$, if $G$ is a finite group and $\chi\in R_\C(G)$ is a virtual character of $G$, then $s(\chi)\in R_\C(G)$ is (well) defined by
$$\forall g\in G,\;\big(s(\chi)\big)(g)=\chi(g^{\pi_n(s)}),$$
where $n$ is any multiple of the exponent of $G$.

The set $\Loc(G)$ is an abelian group, for pointwise addition of maps. If $H$ is a finite group, and $U$ is a finite $(H,G)$-biset, we denote by $\Loc(U):\Loc(G)\to \Loc(H)$ the linear map induced by composition with $R_\C(U):R_\C(G)\to R_\C(H)$.
\begin{lema} With these definitions, the assignment $G\mapsto \Loc(G)$ is a biset functor.
\end{lema}
\pf This is straightforward.  \endpf
\begin{lema}
 \begin{enumerate}
\item Let $G,H,K$ be finite groups. For $\chi\in R_\C(G)$, $f\in \Loc(H)$, and $\rho\in R_\C(K)$, let $\chi\times f\times \rho:\hZt\to R_\C(G\times H\times K)$ be the map defined by
$$\forall s\in \hZt,\;(\chi\times f\times \rho)(s)=s(\chi)\times f(s)\times \rho.$$
Then $\chi\times f\times \rho\in \Loc(G\times H\times K)$.
\item The products $(\chi,f,\rho)\in R_\C(G)\times \Loc(H)\times R_\C(K)\mapsto \chi\times f\times \rho\in \Loc(G\times H\times K)$ endow the biset functor $G\mapsto \Loc(G)$ with a structure of $(R_\C,R_\C)$-bimodule.
\end{enumerate}
\end{lema}
\pf 1. Since $f$ is locally constant, there exists $n\in\N$ and $\sur{f}:(\Z/n\Z)^\times\to R_\C(H)$ such that  $f=\sur{f}\circ\pi_n$. Let $m$ be the exponent of $G$, and $l=n\vee m$ be the least common multiple of $n$ and $m$. Then $f=\sur{f}\circ\pi_{l,n}\circ\pi_l$ and $s(\chi)$ only depends on $\pi_l(s)$. It follows that the map $\chi\times f\times \rho$ factors through $(\Z/l\Z)^\times$, so it is locally constant.\mpn
2. We have to check that the products $(\chi,f,\rho)\mapsto \chi\times f\times \rho$ are associative and unital, and commute with the biset operations. All these verifications are straightforward.\endpf

\begin{nota} Let $G$ be a finite group. We denote by $\lambda_G$ the map $\Loc(G)\to R_\C(G)$ sending $f\in R_\C(G)$ to its value $f(1)$ at the identity element of $\hZt$.
\end{nota} 
\begin{lema} The maps $\lambda_G:\Loc(G)\to R_\C(G)$ define an epimorphism of $(R_\C,R_\C)$-bimodules $\lambda:\Loc\to R_\C$.
\end{lema}
\pf By definition of the biset functor structure on $\Loc$, the maps $\lambda_G$ form a morphism of biset functors $\Loc\to R_\C$. This morphism is clearly surjective, since for any finite group $G$ and any $\chi\in R_\C(G)$, the constant map $\gamma_\chi:s\in \hZt\mapsto \chi\in R_\C(G)$ is locally constant, and such that $\lambda_G(\gamma_\chi)=\chi$. Moreover, for finite groups $G,H,K$, and $(\chi,f,\rho)\in R_\C(G)\times \Loc(H)\times R_\C(K)$, we have that
$$\lambda_{G\times H\times K}(\chi\times f\times\rho)=(\chi\times f\times\rho)(1)=1(\chi)\times f(1)\times \rho=\chi\times\lambda_H(f)\times\rho,$$
so $\lambda$ is a morphism of $(R_\C,R_\C)$-bimodules.\endpf
\begin{prop} \label{R_C non-separable}Let $L$ denote the kernel of $\lambda:\Loc\to R_\C$. Then:
\begin{enumerate}
\item $L$ is an $(R_\C,R_\C)$-bimodule.
\item For a finite group $G$, let $d_G:R_\C(G)\to L(G)$ be the map defined by
$$\forall \chi\in R_\C(G), \forall s\in \hZt,\;d_G(\chi)(s)=s(\chi)-\chi\in R_\C(G).$$
Then the maps $d_G$ form a non-inner derivation $R_C\to L$.
\item In particular $\HH^1(R_\C,L)(\un)\neq 0$, so the first Hochschild cohomology functor $\HH^1(R_\C,L)$ is non zero, and the Green biset functor $R_\C$ is not separable.
\end{enumerate}
\end{prop}

\pf 1. This is clear, as $L$ is the kernel of a morphisms of $(R_\C,R_\C)$-bimodules.\mpn
2. First of all $d_G(\chi)(1)=1(\chi)-\chi=0$, so $d_G$ lands in $L(G)$. Then for a finite group~$K$ and for $\rho\in R_\C(K)$, we have
\begin{align*}
\forall s\in\hZt,\;d_{G\times K}(s)&=s(\chi\times \rho)-(\chi\times \rho)\\
&=s(\chi)\times s(\rho)-(\chi\times\rho)\\
&=s(\chi)\times \big(s(\rho)-\rho\big)+\big(s(\chi)-\chi\big)\times\rho\\
&=\big(\chi\times d_K(\rho)\big)(s) - \big(d_G(\chi)\times\rho\big)(s),
\end{align*}
so $d_{G\times K}(\chi\times\rho)=\chi\times d_K(\rho)+d_G(\chi)\times\rho$, and the maps $d_G:R_\C(G)\to L(G)$ form a derivation $d:R_\C\to L$. \par
This derivation is inner if and only if there exists $m\in L(\un)$ such that 
$$d_G(\chi)=\chi\times m-m\times\chi,$$
for any finite group $G$ and any $\chi\in R_\C(G)$. So $m\in\Loc(\un)$ is a map $\hZt\to R_\C(\un)=\Z$ such that $m(1)=0$, and
\begin{align*}
d_G(\chi)(s)=s(\chi)-\chi&=(\chi\times m)(s)-(m\times \chi)(s)\\
&=s(\chi)\times m(s)-m(s)\times\chi\\
&=m(s)\big(s(\chi)-\chi\big),
\end{align*}
for any finite group $G$, any $\chi\in R_\C(G)$, and any $s\in\hZt$. It follows that $m(s)=1$ if there exists a finite group $G$ and $\chi\in R_\C(G)$ such that $s(\chi)\neq \chi$. \par
Now if $s\neq 1$, there exists $n\in\N$ such that $\pi_n(s)\neq 1$, and we can assume $n\geq 3$. We take for $G$ the cyclic group $C_n$ of $n$-th roots of unity in $\C$, and for $\chi$ the inclusion $i_n:C_n\hookrightarrow \C^\times$. Then we have clearly $s(\chi)\neq\chi$, since $n\geq 3$, so $m(s)=1$. This proves that $m(s)=1$ for $s\neq 1$. But since $m$ is locally constant, there exists $n\in\N$ such that $m(s)=m(1)$ for any $s\in \Ker\, \pi_n$. This is a contradiction, as $m(1)=0$ but $\Ker\, \pi_n\neq\{1\}$. It follows that the derivation $d$ is not inner, completing the proof of Assertion~2.\mpn
3. This follows from Assertion 2.\endpf

\subsection{Separability of $RB_G$}

By Corollary 8.4.12 of \cite{biset}, we know that $RB_G\otimes RB_H$ is isomorphic to $RB_{GH}$ as biset functors, we now give an explicit isomorphism between these functors.

\begin{lema}
\label{eliso}
Let $G$, $H$ and $K$ be finite groups. We define,
\begin{displaymath}
\phi_K:(RB_G\otimes RB_H)(K)\rightarrow RB_{GH}(K),\quad [\alpha\otimes \beta]_{D,D',f}\mapsto f\circ (\alpha\times \beta)
\end{displaymath}
for $D$ and $D'$ finite groups, $\alpha\in RB(D,\,G)$, $\beta\in RB(D',\,H)$ and $f\in RB(K,\, DD')$. On the other direction, 
\begin{displaymath}
\varphi_K:RB_{GH}(K)\rightarrow (RB_G\otimes RB_H)(K),\quad u\mapsto [G\otimes H]_{G,H,u},
\end{displaymath}
for $u\in RB(K,\, GH)$. Then $\phi$ and $\varphi$ define natural isomorphisms between $RB_G\otimes RB_H$ and $RB_{GH}$.
\end{lema}
\pf
It is straightforward to see that $\phi_K$ is well defined and that $\phi$ and $\varphi$ are natural transformations. Now $\varphi_K\circ\phi_K$ gives
\begin{displaymath}
[\alpha\otimes \beta]_{D,D',f}\mapsto f\circ (\alpha\times \beta)\mapsto [G\otimes H]_{G,H,f\circ(\alpha\times \beta)},
\end{displaymath}
but
\begin{displaymath}
[G\otimes H]_{G,H,f\circ(\alpha\times \beta)}=[B_G(\alpha)(G)\otimes B_H(\beta)(H)]_{D,D',f}=[\alpha\otimes\beta]_{D,D',f}.
\end{displaymath}
On the other hand, $\phi_K\circ\varphi_K$ gives
\begin{displaymath}
u\mapsto [G\otimes H]_{G,H,u}\mapsto u\circ(G\times H)=u.
\end{displaymath}
Hence $\phi_K$ and $\varphi_K$ are mutual inverses.
\endpf

Given $A$ and $C$, Green biset functors, the action of $A$ on $A\otimes C$ is given by the composition of arrows
\begin{displaymath}
\Lambda: A\otimes(A\otimes C)\cong (A\otimes A)\otimes C\rightarrow A\otimes C,
\end{displaymath}
where the last arrow is $\mu_{A}\otimes C$. Using the results following Remark \ref{ten-fle}, in Section 2, we have that $\Lambda_K$ behaves in the following way, 
\begin{displaymath}
\begin{array}{rcl}
\big[a\otimes [b\otimes c]_{X_1,Y_1,f_1}\big]_{X,Y,g}&\mapsto & \big[[a\otimes b]_{X,X_1,XX_1}\otimes c\big]_{XX_1,Y_1, g\circ(Xf_1)}\\
&\mapsto & \Big[(\mu_A)_{XX_1}\big([a\otimes b]_{X,X_1,XX_1}\big)\otimes c\Big]_{XX_1,Y_1, g\circ(Xf_1)}\\
&\mapsto & \big[(a\times b)\otimes c\big]_{XX_1,Y_1,g\circ(Xf_1)}
\end{array}
\end{displaymath}
for $a\in A(X)$, $b\in A(X_1)$, $c\in C(Y_1)$, $f_1\in RB(Y,\, X_1Y_1)$ and $g\in RB(K,\, XY)$. Using bilinear maps, this translates, for groups $L$ and $K$, as 
\begin{displaymath}
A(L)\times (A\otimes C)(K)\rightarrow (A\otimes C)(L\times K)
\end{displaymath} 
\begin{displaymath}
\big(\alpha,\, [\beta\otimes \gamma]_{D,D',f}\big)\mapsto \Lambda_{L\times K}\Big(\big[\alpha\otimes [\beta\otimes \gamma]_{D,D',f}\big]_{L,K,LK}\Big), 
\end{displaymath}
which is equal to $[(\alpha\times\beta)\otimes \gamma]_{LD,D', Lf}$, for $\alpha\in A(L)$, $\beta\in A(D)$, $\gamma\in C(D')$ and $f\in RB(K,\, DD')$. A similar description can be given for the right action of $C$. Now we apply this to $A=RB_G$ and $C=RB_H$. 

\begin{lema}
\label{acciones-trasl}
Let $G$, $H$, $K$ and $L$  be finite groups. The left action of $RB_G$ and the right action of $RB_H$ in $RB_G\otimes RB_H$ are given by the natural actions
\begin{displaymath}
RB_G(L)\times RB_{GH}(K)\rightarrow RB_{GH}(L\times K),\quad (\alpha,\,\delta)\mapsto \alpha\times^G \delta, 
\end{displaymath}
and
\begin{displaymath}
 RB_{GH}(K)\times RB_H(L)\rightarrow RB_{GH}(K\times L),\quad (\delta,\,\beta)\mapsto \delta\times^H\beta. 
\end{displaymath}
\end{lema}
\pf
We prove the result only for the left action. By the lines preceding the lemma, we have that the action of $RB_G(L)$ in $(RB_G\otimes RB_H)(K)$ is given by
\begin{displaymath}
RB_G(L)\times (RB_G\otimes RB_H)(K)\rightarrow (RB_G\otimes RB_H)(L\times K)
\end{displaymath} 
\begin{displaymath}
(\alpha,\, [\beta\otimes \gamma]_{D,D',f})\mapsto [(\alpha\times^G\beta)\otimes \gamma]_{LD,D', Lf}
\end{displaymath}
for $\alpha\in RB_G(L)$, $\beta\in RB_G(D)$, $\gamma\in RB_H(D')$ and $f\in RB(K,\, DD')$. Next we apply the morphisms of Lemma \ref{eliso} to translate this for
\begin{displaymath}
RB_G(L)\times RB_{GH}(K)\rightarrow RB_{GH}(L\times K).
\end{displaymath}
If $\delta\in RB_{GH}(K)$, following the composition of morphisms, we have
\begin{displaymath}
(\alpha,\,\delta)\mapsto (\alpha, [G\otimes H]_{G,H,\delta})\mapsto[(\alpha\times^GG)\otimes H]_{LG,H,L\delta}\mapsto L\delta\circ ((\alpha\times^GG)\times H)
\end{displaymath}
with the composition made over $LGH$. Let us see that $L\delta\circ ((\alpha\times^GG)\times H)$ is isomorphic to $(\alpha\times^G\delta)$. For simplicity, suppose that $\delta$ is a $(K,\, GH)$-biset, so that $L\times \delta$ is an $(LK,\, LGH)$-biset, and that $\alpha$ is an $(L,\, G)$-biset, so that $\alpha\times^GG$ is an $(LG,\, G)$-biset where the $G$ on the right acts diagonally on the set $\alpha \times G$. With this, it is easy to see that sending an element $[(l,\,d), (a,\,g,\, h)]$ of $(L\times\delta)\circ ((\alpha\times^GG)\times H)$ to $(la,\, dgh)\in \alpha \times \delta$ gives an isomorphism of $(LK,\, GH)$-bisets.
\endpf
\begin{nota}
The notation used in the previous lemma may be confusing if $H=G$. In this case we will use the following notation.
\begin{displaymath}
RB_G(L)\times RB_{GG}(K)\rightarrow RB_{GG}(L\times K),\quad (\alpha,\,\delta)\mapsto \alpha\times^{G_1} \delta, 
\end{displaymath}
and
\begin{displaymath}
 RB_{GG}(K)\times RB_G(L)\rightarrow RB_{GG}(K\times L),\quad (\delta,\,\beta)\mapsto \delta\times^{G_2}\beta. 
\end{displaymath}
If $\alpha$, $\delta$ and $\beta$ are actually sets, this means that $\alpha\times^{G_1} \delta$ is the $LKGG$-set $\alpha\times \delta$ with diagonal action of the first $G$ in $LK\underline{G}G$ and that $\delta\times^{G_2}\beta$ is the $KLGG$-set $\delta\times \beta$ with diagonal action of the second $G$ in $KLG\underline{G}$. 
\end{nota}

\begin{teo}\label{sep}
Let $R$ be a unital commutative ring and $G$ be a finite group. The Green biset functor $RB_G$ is separable if and only if $|G|\in R^{\times}$.
\end{teo}

By Lemma 5.7 in \cite{hoch}, the functor $RB_G$ is separable if and only if there exists $m\in C(RB_G\otimes RB_G)_{RB_G}(\uno)$ such that $(\mu_{RB_G})_{\uno}(m)=\lae$. To prove the theorem we begin with the following lemma. In what follows  $G$ is a finite group. 

\begin{prop}
\label{conmu}
An element $m\in RB_{GG}(\uno)=RB(GG)$ is in $C(RB_{GG})_{RB_G}(\uno)$ if an only if
\begin{displaymath}
m=\sum_{\substack{L\leqslant G\\L \mathrm{mod}.\,G}}\lambda_L\left[\frac{GG}{\Delta(L)}\right] 
\end{displaymath}
for some $\lambda_L\in R$ or equivalently, if and only if there exists $\alpha\in RB(G)$ such that $m=\Ind_{\Delta(G)}^{GG}(\alpha)$.
\end{prop}
\pf
Suppose first that $m\in C(RB_{GG})_{RB_G}(\uno)$. By the previous lemma, this  translates in the following way
\begin{displaymath}
u \times^{G_1} m=m\times^{G_2} u \in RB(LGG),
\end{displaymath}
for every $u\in RB(LG)$ and every finite group $L$. In particular, this should hold for $L=G$ and $u$ being (the class of) $G$ in $RB(GG)$. Suppose now that $m$ is a $GG$-set. Let us see that $G \times^{G_1} m$ is isomorphic to $\bind_{\Delta(G)G}^{GGG}(m)$, omitting the obvious isomorphism between $GG$ and $\Delta(G)G$. This induction is equal to the composition $(GGG)\circ m$, with $GGG$ seen as a $(GGG, GG)$-biset in an obvious way. Then, it is easy to see that the map
\begin{displaymath}
(GGG)\circ m\rightarrow G\times^{G_1}m,\quad [(g_1,\,g_2,\, g_3),\, x]\mapsto (g_1g_2^{-1},\, (g_2,\, g_3)x)
\end{displaymath}
with $x\in m$ is an isomorphism of $GGG$-sets. Extending this observation linearly, we have that $G\times^{G_1}m=\bind_{\Delta (G)G}^{GGG}(m)$ for any $m\in RB(GG)$. In an analogous way we have $m\times^{G_2} G=\bind_{D_{1,3}(G)}^{GGG}(m)$, where 
\begin{displaymath}
D_{1,3}(G)=\{(g,\,g',\,g)\mid (g,g')\in G\}\leqslant GGG.
\end{displaymath}
So, we are looking for elements $m=\sum \lambda_H (GG)/H$  in $RB(GG)$ (we omit the brackets for simplicity) such that
\begin{displaymath}
\sum \lambda_H\bind_{\Delta (G)G}^{GGG}\left(\frac{GG}{H}\right)=\sum \lambda_H\bind_{D_{1,3}(G)}^{GGG}\left(\frac{GG}{H}\right).
\end{displaymath}
Notice that 
\begin{displaymath}
\bind_{\Delta (G)G}^{GGG}((GG)/H)=\bind_{\Delta (G)G}^{GGG}\circ \bind_{H}^{GG}(\{\bullet\})=\bind_{\overline{H}}^{GGG}(\{\bullet\}),
\end{displaymath}
with
\begin{displaymath}
\overline{H}=\{(a,\,a,\,b)\mid (a,\, b)\in H\}\leqslant GGG.
\end{displaymath}
Hence, doing analogous calculations for $\bind_{D_{1,3}(G)}^{GGG}((GG)/H)$, the previous equality of sums implies that for any $H\leqslant GG$ appearing in $m$, there exists $K\leqslant GG$ and an element $(x,\,y,\,z)\in GGG$ such that
\begin{displaymath}
\overline{H}^{(x,\,y,\,z)}=\{(d,\,c,\,d)\mid (c,\, d)\in K\}.
\end{displaymath}
In consequence, for any $(a,\,b)\in H$, we have $a=b^{zx^{-1}}$. Hence, $H$ is a conjugate of a subgroup of $\Delta(G)$.

Now suppose that
\begin{displaymath}
m=\sum_{\substack{L\leqslant G\\L \,\mathrm{mod}.\,G}}\lambda_L\left[\frac{GG}{\Delta(L)}\right] 
\end{displaymath}
and take  a $KG$-set $\alpha$ in $B_G(K)$. Then $\alpha \times^{G_1} (GG/\Delta(L))\in B_{GG}(K)$ is  the $KGG$-set $\alpha\times(GG/\Delta(L))$ with diagonal action of the first $G$ in $KGG$. Taking the isomorphism interchanging the first and the second $G$ in $KGG$ leaves $GG/\Delta (L)$ unchanged and $\alpha\times(GG/\Delta(L))$ becomes a $KGG$-set with diagonal action of the second $G$ in $KGG$. This is clearly isomorphic to $(GG/\Delta(L))\times^{G_2}\alpha$. 
\endpf

\begin{lema} After identification of $RB_G\otimes RB_G$ with $RB_{GG}$, and evaluation at $\uno$, the map
$$\xymatrix@C=4ex{RB(GG)\cong RB_{GG}(\uno)\cong (RB_G\otimes RB_G)(\uno)\ar[rr]^-{(\mu_{RB_G})_\uno}&&RB_G(\uno)\cong RB(G)}$$
is just  $\Res^{GG}_{\Delta (G)}$.
\end{lema}
\pf Take a $GG$-biset $X$ in $RB(GG)$. Under the isomorphism of Lemma \ref{eliso}, it maps to $[G\otimes G]_{G,G,X}\in (RB_G\otimes RB_G)(\uno)$, where $X$ is seen as an arrow from $GG$ to $\uno$. By Lemma \ref{acciones}, under $(\mu_{RB_G})_\uno$ this element maps to $RB_G(X)(GG)=RB(X\times G)(GG)$. Here $GG$ is in $RB_G(GG)$ as a $GGG$-set with diagonal action of the third $G$, that is
\begin{displaymath}
(g_1,\,g_2,\,g_3)(g,\,g')=(g_1gg_3^{-1},\,g_2g'g_3^{-1}) 
\end{displaymath}
for  $(g_1,\,g_2,\,g_3)\in GGG$ and $(g,\,g')\in GG$. Since $X$ is a $(\uno,\, GG)$-biset, $X\times G$ is a $(G,\, GGG)$-biset with the right action of $GGG$ given by
\begin{displaymath}
(x,\,t)(g_1,\,g_2,\, g_3)= (x(g_1,\, g_2),\,tg_3)
\end{displaymath} 
for $(g_1,\,g_2,\,g_3)\in GGG$ and $(x,\,t)\in X\times G$. Now, $RB(X\times G)(GG)=(X\times G)\circ (GG)$ with the composition being made over $GGG$. It is easy to see that the map
\begin{displaymath}
(X\times G)\circ (GG)\rightarrow \Res^{GG}_{\Delta(G)}(X),\quad [(x,\, t), (g,\, g')]\mapsto x(gt^{-1}, g't^{-1})
\end{displaymath}
is an isomorphism of $G$-sets (under our assumptions $\Res^{GG}_{\Delta(G)}(X)$ is a right $G$-set, but this action can clearly be written as a left action). Extending linearly we obtain the result for any element of $RB(GG)$.
\endpf

\noindent{\bf Proof of Theorem~\ref{sep}:} By the previous lemma and proposition, the functor $RB_G$ is separable if and only if there exists $\alpha\in RB(G)$ such that the element $m=\Ind_{\Delta (G)}^{GG}(\alpha)\in RB(GG)$ 
 is mapped to the identity of $RB(G)$ by restriction to the diagonal.  Since $\Delta(G)\dom GG/\Delta(G)\cong \{(g,1)\mid g\in Cl(G)\}$, where $Cl(G)$ is a set of representatives of conjugacy classes of $G$, and since $\Delta(G)\cap {^{(g,1)}\Delta(G)}=\Delta\big(C_G(g)\big)$, by the Mackey formula, we have that
$$\Res_G^{GG}\Ind_G^{GG}\alpha=\sum_{g\in Cl(G)}\Ind_{C_G(g)}^G\Res_{C_G(g)}^G\alpha=\Gamma_G\cdot\alpha,$$
where $\Gamma_G=\sum_{g\in Cl(G)}\limits[G/Cl(g)]$. In other words $RB_G$ is separable if and only if $\Gamma_G$ is invertible in $RB(G)$. But the $G$-set $\Gamma_G$ is nothing but the set $G$, acted on by $G$ by conjugation. If $\Gamma_G$ is invertible, with inverse $\alpha$, then in particular 
$$\Phi_\un(\Gamma\cdot\alpha)=|\Gamma_G||\alpha|=\Phi_\un(G/G)=|G/G|=1$$
in $R$, so $|\Gamma_G|=|G|$ is invertible in $R$.\par
Conversely, if $|G|$ is invertible in $R$, then we have idempotents
$$e_H^G=\frac{1}{|N_G(H)|}\sum_{K\leqslant H}|K|\mu(K,H)[G/K]$$
in the ring $RB(G)$, and 
$$\Gamma_G=\sum_{\substack{H\leqslant G\\H \,\mathrm{mod.}\, G}}\Gamma_G\cdot e_H^G=\sum_{\substack{H\leqslant G\\H \,\mathrm{mod.}\, G}}|(\Gamma_G)^H|e_H^G=\sum_{\substack{H\leqslant G\\H \,\mathrm{mod.}\, G}}|C_G(H)|e_H^G.$$
 Now since $|C_G(H)|$ is also invertible in $R$ for all $H\leqslant G$, we can set 
$$\alpha=\sum_{\substack{H\leqslant G\\H \,\mathrm{mod.}\, G}}\frac{1}{|C_G(H)|}e_H^G$$
in $RB(G)$, and this element $\alpha$ is inverse to $\Gamma_G$. So $RB_G$ is separable.\endpf

\section{Link with evaluations}\label{evaluations}

Let $A$ be a Green biset functor, and $M$ be an $(A,A)$-bimodule. Then (see Section~2 of~\cite{greenfields}), for any finite group $G$, the evaluation $A(G)$, endowed with the ``dot'' product, becomes a ring, and the evaluation $M(G)$ becomes an $\big(A(G),A(G)\big)$-bimodule. It is then natural to try to compare the evaluations $\CH H^n(A,M)$ of the Hochschild cohomology functors of $A$ with values in $M$, and the ``ordinary'' Hochschild cohomology groups $HH^n\big(A(G), M(G)\big)$, for a given finite group $G$. Similarly, it is natural to compare the property of being separable for the Green biset functor $A$ and being separable for the ring $A(G)$.\par
In this section, we give some examples of such comparisons. In particular, we show  (Theorem~\ref{burnside} for $R=\Z$) that if $G$ is a non-trivial finite group, then its Burnside ring $B(G)$ is not separable.  In contrast, the Green biset functor $B$ is separable, since the multiplication morphism $B\otimes B\to B$ is an isomorphism, as $B$ is the identity object for the tensor product of biset functors. However, we also show  (Proposition~\ref{burnside2} for $R=\Z$) that the first Hochschild cohomology group $HH^1\big(B(G),B(G)\big)$ is always trivial.\medskip\par


\begin{teo} \label{burnside}Let $R$ be a unital commutative ring and $G$ be a finite group. Then $RB(G)$ is a separable $R$-algebra if and only if the order of $G$ is invertible in $R$
\end{teo}
\pf Suppose first that the $R$-algebra $RB(G)$ is separable. Equivalently, there exists an element $u\in RB(G)\otimes_R RB(G)$ such that $\alpha\cdot u=u\cdot\alpha$ for any $\alpha\in RB(G)$, and $\mu_{RB(G)}(u)=[G/G]$, where $\mu_{RB(G)}:RB(G)\otimes_R RB(G)\to RB(G)$ is the product map. Since $\mu_{RB(G)}$ is a morphism of $\big(RB(G),RB(G)\big)$-bimodules, we have in particular that
$$\mu_{RB(G)}\big([G/\un]\cdot u\big)=[G/\un]\cdot \mu_{RB(G)}(u)=[G/\un]\cdot [G/G]=[G/\un].$$
Moreover $[G/\un]\cdot u=u\cdot [G/\un]$. Now the elements $[G/H]\otimes [G/K]$, for $H$ and $K$ in $[s_G]$, form an $R$-basis of $RB(G)\otimes_R RB(G)$. Since $[G/\un]\cdot [G/H]=|G:H|[G/\un]$ for any $H\in [s_G]$, the product $[G/\un]\cdot u$ is a linear combination of elements $[G/\un]\otimes [G/K]$, for $K\in [s_G]$. But it is equal to $u\cdot [G/\un]$, which similarly is a linear combination of elements $[G/H]\otimes [G/\un]$, for $H\in[s_G]$. It follows that $[G/\un]\cdot u=r\,[G/\un]\otimes [G/\un]$, for some $r\in R$. Then
$$[G/\un]=\mu_{RB(G)}\big([G/\un]\cdot u\big)=r\,\mu_{RB(G)}\big([G/\un]\otimes [G/\un]\big)=r\,[G/\un]^2=r|G|[G/\un].$$
Then $r|G|=1$, so $|G|$ is invertible in $R$, as was to be shown.
\medskip\par
Conversely, suppose that the order of $G$ is invertible in $R$, and consider the element
$$u:=\sum_{H\in [s_G]}e_H^G\otimes e_H^G$$
of $RB(G)\otimes RB(G)$. Then for any finite $G$-set $X$
\begin{align*}
X\cdot u&=\sum_{H\in [s_G]}(X\cdot e_H^G)\otimes e_H^G=\sum_{H\in [s_G]}|X^H|e_H^G\otimes e_H^G\\
&=\sum_{H\in [s_G]}e_H^G\otimes |X^H|e_H^G=\sum_{H\in [s_G]}e_H^G\otimes (e_H^G\cdot X)=u\cdot X.
\end{align*}
Moreover, the image of $u$ by the product map $\mu_{RB(G)}:RB(G)\otimes_R RB(G)\to RB(G)$ is equal to
$$\sum_{H\in [s_G]}(e_H^G)^2=\sum_{H\in [s_G]}e_H^G=[G/G].$$
It follows that $u$ is a {\em separability element} of $RB(G)\otimes RB(G)$, so $RB(G)$ is separable. This completes the proof.\endpf

\begin{prop} \label{burnside2} Let $R$ be a unital commutative ring and $G$ be a finite group. If $R$ has no $|G|$-torsion, then the only derivation of $RB(G)$ is zero, i.e. $HH^1\big(RB(G),RB(G)\big)=0$.
\end{prop}
\pf For $H\in[s_G]$, let $v_H=|G|e_H^G$. Then $v_H\in RB(G)$, and $v_H\cdot \alpha=\Phi_H(\alpha)\,v_H$ for any $\alpha\in RB(G)$, where $\Phi_H:RB(G)\to R$ is the unique $R$-linear map sending a finite $G$-set $X$ to the number $|X^H|$ of $H$-fixed points on $X$. In particular $(v_H)^2=|G|v_H$, since $\Phi_H(v_H)=|G|$.\par
Let $d:RB(G)\to RB(G)$ be a derivation. Then
$$d\big((v_H)^2\big)=2v_H\cdot d(v_H)=2\Phi_H\big(d(v_H)\big)\,v_H=|G|\,d(v_H).$$
Applying $\Phi_H$ to the last equality, we get that 
$$2\Phi_H\big(d(v_H)\big)\,\Phi_H(v_H)=2|G|\,\Phi_H\big(d(v_H)\big)=|G|\,\Phi_H\big(d(v_H)\big).$$
It follows that $|G|\Phi_H\big(d(v_H)\big)=0$, so $\Phi_H\big(d(v_H)\big)=0$ since $R$ has no $|G|$-torsion. Hence $|G|d(v_H)=2\Phi_H\big(d(v_H)\big)\,v_H=0$, so $d(v_H)=0$, since $RB(G)$ has no $|G|$-torsion either. Now if $\alpha\in RB(G)$, then $|G|\alpha$ is a linear combination of the elements $v_H$, for $H\in[s_G]$. It follows that $|G|d(\alpha)=0$, so $d(\alpha)=0$ since $RB(G)$ has no $|G|$-torsion. Hence $d=0$, as was to be shown.  \endpf
\begin{rem} If $R$ has $|G|$-torsion, then $RB(G)$ may admit non-trivial derivations: For example, if $G$ has prime order $p$, and $R$ has characteritic $p$, then the algebra $RB(G)$ is isomorphic to $R[X]/(X^2)$, which has a non-trivial derivation sending $aX+b$ to $aX$, for $a,b\in R$.
\end{rem}


\centerline{\rule{5ex}{.1ex}}
\begin{flushleft}
Serge Bouc, CNRS-LAMFA, Universit\'e de Picardie, 33 rue St Leu, 80039, Amiens, France.\\
{\tt serge.bouc@u-picardie.fr}\vspace{1ex}\\
Nadia Romero, DEMAT, UGTO, Jalisco s/n, Mineral de Valenciana, 36240, Guanajuato, Gto., Mexico.\\
{\tt nadia.romero@ugto.mx}
\end{flushleft}

\end{document}